\documentclass[10pt]{article}
\usepackage{amssymb}
\usepackage{graphicx}
\usepackage{xcolor} 
\usepackage{tensor}
\usepackage{fullpage} 
\usepackage{amsmath}
\usepackage{amsthm}
\usepackage{verbatim}
\providecommand{\bysame}{\leavevmode\hbox to3em{\hrulefill}\thinspace}
\providecommand{\MR}{\relax\ifhmode\unskip\space\fi MR }

\providecommand{\href}[2]{#2}
\DeclareFontFamily{U}{mathx}{\hyphenchar\font45}
\DeclareFontShape{U}{mathx}{m}{n}{
	<5> <6> <7> <8> <9> <10>
	<10.95> <12> <14.4> <17.28> <20.74> <24.88>
	mathx10
}{}
\DeclareSymbolFont{mathx}{U}{mathx}{m}{n}
\DeclareFontSubstitution{U}{mathx}{m}{n}
\DeclareMathAccent{\widecheck}{0}{mathx}{"71}
\DeclareMathAccent{\wideparen}{0}{mathx}{"75}

\usepackage{enumitem}
\setlist[enumerate]{leftmargin=1.5em}
\setlist[itemize]{leftmargin=1.5em}

\usepackage{seqsplit}

\definecolor{green}{rgb}{0,0.8,0} 

\newtheorem{maintheorem}{Theorem}

\newtheorem{theorem}{Theorem}[section]
\newtheorem{corollary}[theorem]{Corollary}
\newtheorem{lemma}[theorem]{Lemma}
\newtheorem{proposition}[theorem]{Proposition}

\theoremstyle{definition}

\theoremstyle{remark}
\newtheorem{remark}[theorem]{Remark}

\numberwithin{equation}{section}

\makeatletter

\newcommand{\be}{\begin{equation}}
\newcommand{\ee}{\end{equation}}
\renewcommand{\a}{\alpha}

\newcommand{\yy}{\mathcal{Y}^1} 
\newcommand{\nrm}{\@ifstar{\nrmb}{\nrmi}}
\newcommand{\nrmi}[1]{\Vert{#1}\Vert}
\newcommand{\nrmb}[1]{\left\Vert{#1}\right\Vert}
\newcommand{\abs}{\@ifstar{\absb}{\absi}}
\newcommand{\absi}[1]{\vert{#1}\vert}
\newcommand{\absb}[1]{\left\vert{#1}\right\vert}
\newcommand{\brk}{\@ifstar{\brkb}{\brki}}
\newcommand{\brki}[1]{\langle{#1}\rangle}
\newcommand{\brkb}[1]{\left\langle{#1}\right\rangle}
\newcommand{\set}{\@ifstar{\setb}{\seti}}
\newcommand{\seti}[1]{\{#1\}}
\newcommand{\setb}[1]{\left\{ #1\right\}}
\makeatother

\newcommand{\nnrm}[1]{{\vert\kern-0.25ex\vert\kern-0.25ex\vert #1 
    \vert\kern-0.25ex\vert\kern-0.25ex\vert}}

\newcommand{\VERT}[1]{{\left\vert\kern-0.25ex\left\vert\kern-0.25ex\left\vert #1 
    \right\vert\kern-0.25ex\right\vert\kern-0.25ex\right\vert}}

\newcommand{\lap}{\Delta}

\newcommand{\ud}{\mathrm{d}}
\newcommand{\rd}{\partial}
\newcommand{\nb}{\nabla}

\newcommand{\alp}{\alpha}

\newcommand{\gmm}{\gamma}
\newcommand{\Gmm}{\Gamma}

\newcommand{\Lmb}{\Lambda}

\newcommand{\tht}{\theta}

\newcommand{\Omg}{\Omega}

\newcommand{\bbR}{\mathbb R}

\newcommand{\bbT}{\mathbb T}

\newcommand{\calX}{\mathcal X}
\newcommand{\calY}{\mathcal Y}

\begin{document}

\title{Active vector models generalizing 3D Euler \\
and electron--MHD equations}
\author{Dongho Chae\thanks{Department of Mathematics, Chung-Ang University, Heuk-Seok Ro 84,
		Seoul, Republic of Korea 156-756. E-mail: dchae@cau.ac.kr} \and In-Jee Jeong\thanks{Department of Mathematical Sciences and RIM, Seoul National University, Gwanak-gu, Gwanak 1, Seoul, Republic of Korea 08826. E-mail: injee\_j@snu.ac.kr}}

\date{\today}



\maketitle


\begin{abstract}
We introduce an active vector system, which generalizes both the 3D Euler equations and the electron--magnetohydrodynamic equations (E--MHD). We may as well view the system as singularized systems  for  the 3D Euler equations, in which case 
the equations of  (E--MHD) correspond to the  order two more singular one than the 3D Euler equations. The generalized  surface quasi-geostrophic equation (gSQG) can be also embedded into a special case of  our system when the unknown functions are  constant in one coordinate direction. We investigate some basic properties of this system as well as the conservation laws. In the case when the system corresponds up  to order one more singular than the 3D Euler equations, we prove local well-posedness in the standard Sobolev spaces. The proof crucially depends on a sharp commutator estimate similar to the one used for (gSQG) in \cite{CCCGW}. Since the system covers many areas of both physically and mathematically interesting cases, one can expect that there are various related problems  to be investigated, parts of which are discussed here.
\end{abstract}

\section{Introduction}

\subsection{The generalized SQG equations}

The \textit{generalized surface quasi-geostrophic} (gSQG) {equation is an active scalar system given in $\mathbb{R}^2$}  as : 
\begin{equation}\label{eq:gSQG} \tag{gSQG}
	\left\{
	\begin{aligned}
		& \rd_t \tht + v \cdot \nb \tht = 0, \\
		& v = \nb^\perp \Gmm \tht, 
	\end{aligned}
	\right.
\end{equation} with $\tht(t,\cdot):\Omg\rightarrow\bbR$, where $\Gmm$ is a Fourier multiplier. In the recent years, the model \eqref{eq:gSQG} together with its dissipative analogues has been intensively studied, with topics including local and global well-posedness (\cite{CCW,CCW2,CCW3,CCCGW,Oh1,Oh2,con1, con2, con2, cfm, Gan,Li-SQG,WJ,KhoRo1,KhoRo2,HSZ1,HSZ2,CNgu,Ngu,ElgindiA,CCGM,CJO1,CJO2,FV,Kwon,Mar}), finite and infinite time singularity formation (\cite{KN,HeKi,KRYZ,KYZ,Den,KS,Z,X,Denisov-merging}), rotating and traveling-wave solutions (\cite{CGI,GOMEZ,CCG-TRAN,HMsqg,CCG-duke,CCG,HHH,CQZZ,CQZZ2,EJS,EJS1,EJS2}), non-uniqueness (\cite{IV,CKL,BSV,Shv}), and so on. For any $\Gmm$, the system describes the transport of a scalar by an incompressible flow. The primary motivation for studying the system \eqref{eq:gSQG} comes from   {the fact that  (gSQG) interpolates the two physically important cases  for  $\Gmm = (-\lap)^{-1}$ and $\Gmm =(-\lap)^{-\frac12}$, where \eqref{eq:gSQG} corresponds to {\em 2D Euler equations} and the {\em surface quasi-geostrophic (SQG) equations}}, respectively. The SQG equations were derived in \cite{CMT1,CMT2} from the system for non-homogeneous three-dimensional half-space rotating fluid under the so-called \textit{quasi-geostrophic} approximation. Various interesting phenomena occur for solutions to \eqref{eq:gSQG} with different choice of multipliers; see recent surveys \cite{con1,Ki-sur,Ki-sur2} as well as aforementioned references.

\subsection{The active vector system}
In this note, we introduce  {an active vector system}  defined in  {$\Bbb R^3$},  which not only generalizes two physically important systems, namely the \textit{3D Euler  equations} and the \textit{electron--MHD equations} (also referred to as the Hall equations) but also contains the entire family of (gSQG) as sub-systems. The system we suggest is given by \begin{equation}\label{eq:EH}
	\left\{
	\begin{aligned}
		& \rd_t B + \nb\times ( ( \nb\times \Gmm [B] ) \times B) = 0,  \\
		& \nabla \cdot B = 0, \\ 
	\end{aligned}
	\right.
\end{equation} where $B(t,\cdot):\bbR^3\rightarrow\bbR^3$ and $\Gmm$ is a Fourier multiplier whose symbol will be denoted by $\gmm=\gmm(|\xi|)\ge 0$. Since $B$ is divergence-free, \eqref{eq:EH} can be written as an \textit{active vector} form: \begin{equation}\label{eq:EH-active}
\left\{
\begin{aligned}
	& \rd_t B + V \cdot \nb B = B\cdot \nb V,  \\
	& V = -\nb\times \Gmm[B], \\
	& \nabla \cdot B = 0. 
\end{aligned}
\right.
\end{equation} The important special cases $\Gmm = (-\lap)^{-1}$ and $\Gmm = 1$ correspond to  {the 3D Euler equations} (with  {$B=\nabla \times V$} being the vorticity) and the so-called electron magneto-hydrodynamics (E--MHD) system ($B$ being the magnetic field), which can be obtained from the Hall--MHD system by setting the fluid velocity to be zero (\cite{ADFL,BSZ,Light,Pecseli}). That is, the equation for the magnetic field reads \begin{equation}\label{eq:eMHD} \tag{E--MHD}
\left\{
\begin{aligned}
		& \rd_t B + \nb\times ( ( \nb\times B ) \times B) = 0,  \\
	& \nabla \cdot B = 0. \\ 
\end{aligned}
\right.
\end{equation} Recently, the Hall--MHD, E--MHD systems and their viscous counterparts have received a lot of attention in the PDE community (\cite{CDL,ChaeLee,CS,CWW,CWe,Dai1,Dai2,DT1,DT2}), not only because these equations are important in various physical phenomena, but also they naturally appear in  hydrodynamic limit problems (\cite{ADFL}) and raise new technical difficulties. Already at the level of proving well-posedness, \eqref{eq:eMHD} can be considered as a transport equation for $B$ with advecting velocity $V=\nb\times B$, which is one order more singular than $B$. This is in stark contrast with the three-dimensional incompressible Euler equations (\eqref{eq:EH-active}), where $V= \nabla\times (-\lap)^{-1}B$ is one order more regular than $B$. The 3D Euler equations have solutions which blow up in finite and infinite time (\cite{EJE,EJO,ChenHou,Elgindi-3D,CJ-Hill,CJ-axi}). Therefore, it will be interesting to investigate the dynamics of solutions for the interpolating systems, which was precisely the motivation for studying the generalized SQG equations. 
Somewhat surprisingly, it turns out that if we consider the special case of \eqref{eq:EH-active} with   $B(t,x)=(0,0, \theta (t,x_1, x_2)  ),$ and $ V= \nabla ^\perp \Gmm [\theta]$, then  we obtain precisely (gSQG), see \ref{subsec:2andhalfd} for details. That is, any gSQG equation can be embedded into our system \eqref{eq:EH}, so that \eqref{eq:EH} contains all the phenomena observed in the gSQG case. In the current work, we shall focus on the multipliers $\Gmm$ which lie in between  {those above two cases} and discuss well/ill-posedness of the Cauchy problem for \eqref{eq:EH}. 

\begin{remark}
	The term ``active vector'' was used in \cite{con-EL} to denote transport systems for vector-valued functions associated with the incompressible Euler equations. We generalize the notion of an active vector to allow a stretching term on the right hand side. 
\end{remark}

\subsection{Main result}

Our first main result provides local well-posedness of smooth solutions for \eqref{eq:EH} when the operator $\Gmm$ is up to one order more singular than the case of the 3D Euler equations. The precise assumptions we impose on $\Gmm$ are given as follows. 
\medskip

\noindent \textbf{Assumptions on $\Gmm$}. Let $\Lmb = (-\lap)^{\frac12}$ be the Zygmund operator. We shall assume that $\Gmm$ satisfies    \begin{itemize}
	\item  {$\Gmm \lesssim \max\{\Lmb^{-1}, \Lmb^{-2}\}$}, $\Gmm$ is decreasing with $\Lmb |\Gmm'| \lesssim \Gmm$: to be precise, for any $|\xi|>0$, we have \begin{equation}\label{as1}
		\begin{split}
			\gmm(|\xi|) \le  {C(|\xi|^{-1} + |\xi|^{-2} )}
		\end{split}
	\end{equation}
	and  \begin{equation}\label{as2}
		\begin{split}
			\gmm'(|\xi|)\le 0, \qquad |\xi| |\gmm'(|\xi|)| \le C\gmm(|\xi|)
		\end{split}
	\end{equation} for some $C>0$. 
	\item  {$\Gmm \gtrsim \min\{ \Lmb^{-2}, \Lmb^{-1} \}$}: that is, \begin{equation}\label{as3}
		\begin{split}
						\gmm(|\xi|) \ge c \, {\min\{ |\xi|^{-2}, |\xi|^{-1} \}}
		\end{split}
	\end{equation} for all $|\xi|>0$ with some $c>0$. This assumption ensures that $\nb\times \Gmm [B]$ is well-defined as an $L^2$ function for $B \in \dot{H}^{-1}\cap L^{2}(\bbR^3)$. 
\end{itemize}

\medskip

\noindent \textbf{Examples of $\Gmm$}. The assumptions are clearly satisfied by $\Gmm = \Lmb^{-a}$ (fractional powers of the negative Laplacian), with any $1 \le a \le 2$. It is not difficult to check that as long as $1 < a < 2$, one can take less standard symbols $\Gmm = \Lmb^{-a} \log^{\alp_1}(10 + \Lmb)$, $\Lmb^{-a} \log^{\alp_1}(10 + \Lmb) \log^{\alp_2}(10 + \log (10 + \Lmb)) $, etc, for any $\alp_i$.

\medskip

{We now introduce the function space $\mathcal{Y}^1(\bbR^n)$ defined by
$$\mathcal{Y}^1(\bbR^{n})= \left\{ f\in \mathcal{D} ^\prime(\Bbb R^n) \, |\,  \int_{\Bbb R^n} { (1 + |\xi|)} |\hat{f}(\xi )|\,  \ud\xi =: \| f\|_{\mathcal{Y}^1} <+\infty\right \},$$ which is a slight variant of the $\calX^{1}$ space used in \cite{CCGS,lei-lin}.} We are now ready to state our main result.
\begin{maintheorem}\label{thm:lwp}
	We have the following local well-posedness and blow-up criterion for \eqref{eq:EH}. 
\begin{itemize}
\item[(i)] (Local well-posedness)	For any $s > \frac{5}{2}$, the system \eqref{eq:EH} is locally well-posed in $H^{s}(\Bbb R^3)$: given a divergence-free initial data $B_0 \in H^{s}$, there exist $T>0$ and a unique solution $B \in C([0,T];H^{s}(\bbR^3))$ with $B(t=0)=B_{0}$. 
\item[(ii)](Blow up criterion) Assume that $B$ is a solution to \eqref{eq:EH} belonging to $C([0,T];H^{s}(\bbR^3))$. Then, 
$$\limsup_{ t\to T} \|B\|_{H^s} =+\infty\quad \text{ if and only if }\quad \int_0 ^T  { \|B\|_{\yy} }\, \ud t =+\infty.
$$ Therefore, the solution can be continued past $T$ if and only if $\int_0 ^T  {\|B\|_{\yy} }\, \ud t <\infty$. 
\end{itemize}
\end{maintheorem}

\medskip 

 {
When the multiplier $\Gmm$ is more singular, we can prove local well-posedness under the presence of a dissipative term. For concreteness, we focus on the case $\Gmm = \Lmb^{-a}$ with $0 \le a < 1$, whose local well-posedness is not covered by Theorem A in the inviscid case.
\begin{maintheorem}\label{thm:lwp2}
	Consider the fractionally dissipative system \begin{equation}\label{eq:EH-diss}
		\left\{
		\begin{aligned}
			& \rd_t B + \nb\times ( ( \nb\times \Lmb^{-a}[B] ) \times B) = -\Lmb^{b}B,  \\
			& \nabla \cdot B = 0. 
		\end{aligned}
		\right.
	\end{equation} Then, for $0 \le a <1$ and $1-a <b $, the Cauchy problem for \eqref{eq:EH-diss} is locally well-posed in $C([0,T];H^{s}(\bbR^3))$ in the same sense as in Theorem A, as long as $s>\frac{7}{2}-a$. 
\end{maintheorem} }

\medskip

\noindent \textbf{Sharpness of Theorems A and B}. We emphasize that the local well-posedness statements in Theorems \ref{thm:lwp} and \ref{thm:lwp2} are not simple consequences of standard energy estimates, even in the case of $\Gmm = \Lmb^{-a}$. To explain the sharpness of Theorem \ref{thm:lwp}, one can note that already in the case of 3D Euler where $\Gmm=\Lmb^{-2}$,  $\nb V$ in the right hand side of \eqref{eq:EH-active} is already as singular as $B$, so that as soon as $\Gmm=\Lmb^{-a}$ with $a<2$, $\nb V$ is strictly more singular than $B$. Therefore, to close the a priori estimate in Sobolev spaces in the range $a \in [1,2)$, a cancellation structure and sharp commutator estimate must be used, which can be considered as an extension of the one proved in \cite{CCCGW} for local well-posedness of gSQG equations in the ``singular'' regime. 
Furthermore, we believe that the threshold $a=1$ is sharp for local well-posedness; in a recent work \cite{JO1}, \textit{strong ill-posedness} of the Cauchy problem for the electron--MHD system (\eqref{eq:EH} with $\Gmm = 1$) was established in any sufficiently regular Sobolev, H\"older and even Gevrey spaces. From the paper \cite{JO1}, there are many reasons to believe that the loss of derivative in the electron--MHD system is precisely equal to 1, suggesting ill-posedness of \eqref{eq:EH} in the whole range of multipliers $\Gmm = \Lmb^{-a}$ with $0<a<1$. Formally, taking  $\Gmm = \Lmb^{-a}$ and plugging in the ansatz $j \simeq \Lmb^{-1}b^{z}$ (which is reasonable since $b^{z}$ should be comparable with $\nb j$), we have that \begin{equation*}
	\begin{split}
		\rd_t b^z + \nb^\perp\Gmm b^z \cdot \nb b^z - \nb^\perp\Lmb^{-1} b^{z} \cdot \nb \Lmb^{1-a} b^{z} \simeq 0, 
	\end{split}
\end{equation*} which is expected to be ill-posed, due to the singular term $\nb \Lmb^{1-a} b^{z}$. The methods developed in \cite{JO1,JO3,CJO1} should be applicable to give ill-posedness, but the proof will be highly involved since \eqref{eq:EH} is a three-dimensional system involving a non-local operator. Similarly, we expect the restriction $1-a<b$ in Theorem \ref{thm:lwp2} to be sharp; indeed, in the electron--MHD case $a = 0$, while Theorem \ref{thm:lwp2} requires $b>1$, strong ill-posedness was shown whenever $b\le1$ (\cite{JO1}). Furthermore, it can be considered as an extension of the local well-posedness results of \cite{CDL,CWe} in the case $b=2$. 

\medskip

\noindent \textbf{Organization of the paper.} {Theorems A and B will be proved in Section 3 below, after discussing a few basic properties in Section 2. Possible extensions and open problems will be discussed in Section 4.}

\section{Basic properties of the active vector system}

\subsection{Conservation laws}
 
 We can also write  \eqref{eq:EH} in  another velocity formulation.
Let us  consider  $u$, which is the   solution of 
$ \nb\times u = B$   satisfying the ``gauge fixing condition" $ \nb \cdot u = 0$. 
We also assume  $u$ decays  sufficiently fast at infinity. 
Then, since \eqref{eq:EH} is written as
$$ \nabla \times ( \rd_t u + (\nb\times u) \times \lap\Gmm u  )=0, $$
there exists a scalar function $Q=Q(x,t)$ such that 
\be\label{eq:EH-vel}
		\left\{
	\begin{aligned} & \rd_t u + (\nb\times u) \times \lap\Gmm u =-\nb Q ,\\
	 &  \nb \cdot u = 0.
	 	\end{aligned}	\right. 
\ee
We recall that in the case of Euler equations, $Q = p + \frac{|u|^{2}}{2}$, where $p$ is the pressure. 
Comparing \eqref{eq:EH-vel} with \eqref{eq:EH},  we find the following relation between the ``two velocities",
\be
V=\Delta \Gmm [u].
\ee
We present the basic conservation laws for the active vector system. Among others  please  note that the helicity defined below  is independent of the multiplier $\Gmm$ in \eqref{eq:EH}.
\begin{proposition}[Conservation laws]
	For a sufficiently smooth and decaying solutions of \eqref{eq:EH}, we have the following conservations.
	\begin{itemize}
	\item[(i)] {\bf (Energy)}   Let us define the energy:
	$$E(t)= \frac12 \int_{\Bbb R^3} |\Gmm^{\frac12}B |^2 \, \ud x=\frac12\int_{\Bbb R^3} |(-\Delta)^\frac12 \Gmm ^\frac12  u|^2 \, \ud x . $$
	Then, $ E(t)=E(0). $
	\item[(ii)] {\bf (Helicity)}  We define the helicity:
	$$H(t)= \int_{\Bbb R^3} B\cdot  \nabla \times (-\Delta )^{-1} B \, \ud x =  \int_{\Bbb R^3}   B\cdot  u \,\, \ud x. $$
	Then, $ H(t)=H(0)$. 
	\end{itemize}
\end{proposition}
\begin{proof}
\begin{itemize}
\item[(i)]	Taking the $L^2 (\Bbb R^3)$ inner product of \eqref{eq:EH} with $\Gmm B$ and integrating by part, we have that \begin{equation*}
		\begin{split}
		\frac12 \frac{\ud E(t)}{\ud t}  &= - \int_{\Bbb R^3}  \nb\times ( (\nb\times\Gmm B) \times B) \cdot \Gmm B \, \ud x \\&  = - \int_{\Bbb R^3} ( (\nb\times\Gmm B) \times B) \cdot( \nb\times  \Gmm B) \, \ud x = 0. 
		\end{split}
	\end{equation*} 
\item[(ii)] 	 Using \eqref{eq:EH} and \eqref{eq:EH-vel}, and integrating by part we compute
\begin{align*}
\frac{\ud H(t)}{\ud t} &= \int_{\Bbb R^3} B_t\cdot u \, \, \ud x +\int_{\Bbb R^3} B\cdot u_t \, \, \ud x \\
&= -\int_{\Bbb R^3} \nabla \times \left(  (\nabla \times \Gmm[B] )  \times B\right)\cdot u \, \, \ud x - \int_{\Bbb R^3} B\cdot \left( (\nabla \times u) \times \Delta \Gmm[u] + \nabla Q\right) \, \, \ud x\\
&=  -\int_{\Bbb R^3}  \ (\nabla \times \Gmm[B] )  \times B\cdot  B \, \, \ud x - \int_{\Bbb R^3} B\cdot B\times \Delta \Gmm[u] \, \, \ud x 
-\int_{\Bbb R^3} B\cdot \nabla Q \, \ud x\\
&=0.
\end{align*}
\end{itemize} 
\noindent This finishes the proof. 
\end{proof}

\medskip

We  shall observe that the formulation \eqref{eq:EH-active}, viewing $V$ as a velocity field, is useful in  the Lagrangian formulation of \eqref{eq:EH}.
Let us introduce the particle trajectory mapping $X(\a, t)$ generated by  $V$ defined by
$$\frac{\partial X (t,\a)}{ \partial t } = V(t,X(t,\a)) \quad ; \quad X(0,\a)=\a.
$$
As an immediate consequence of \eqref{eq:EH-active} the following  Cauchy formula 
\be\label{cauchy}
B( t,X(t,\a)) = B_0 (\a) \cdot \nabla X(t,\a)
\ee
holds.  For  details of the  proof see e.g. \cite[Lemma 1.4]{MB}, which is obviously applicable to our case.  Let the initial closed curve $\mathcal{C}_0$ be  an integral curve of $B_0$, 
namely there exists $ \lambda (\cdot ) : [0, 1) \to \Bbb R$ such that 
\be\label{c0}
\mathcal{C}_0=\{ \eta (s)\in \Bbb R^3  \, |\, s\in [0, 1], \eta  (0)=\eta (s) , \, \eta^\prime (s) = \lambda (s) B_0 (\eta(s) )\}.
\ee
and define 
\be\label{c1}
 \mathcal{C}_t= X( t,\mathcal{C}_0 )= \{ X(t, \eta (s)) \in \Bbb R^3  \, |\, s\in [0, 1], \eta  (0)=\eta (s) , \, \eta^\prime (s) = \lambda (s) B_0 (\eta(s) )\}.
\ee
Then, we claim $C_t$ is an integral curve of $B(\cdot, t)$, namely 
\be
 \frac {\partial X(t, \eta (s)) }{\partial s} = \lambda (s) B (t, \eta(s) )      \quad \forall s\in[0,1], \quad t>0.
\ee
Indeed,  using  \eqref{cauchy},  we have
\begin{align*} 
\frac {\partial X(t, \eta (s)) }{\partial s}&= \frac{d\eta(s) }{ds} \cdot  \nabla X(t, \a)=\lambda (s) B_0 ( \eta(s) ) \cdot  \nabla X(t,\a)\\
&= \lambda (s) B (t , \eta(s) ) .
\end{align*}  

\subsection{ {$2+\frac12$-dimensional case}} \label{subsec:2andhalfd}

 {
The system \eqref{eq:EH} formally satisfies translational invariance, and therefore we may consider special solutions which are not dependent on the third coordinate. Such a reduction is sometimes referred to as the $2+\frac{1}{2}$--dimensional reduction (\cite{MB}). Both 3D Euler and Hall-MHD systems have been extensively studied under this simplifying ansatz (\cite{BT,DiPM,JY,JO1,Du,Ya-PhysD,CWo1}). From the $z$-independence and divergence-free condition, we can write \begin{equation*}
	\begin{split}
		B = (-\rd_y j, \rd_x j, b^z).
	\end{split}
\end{equation*} For simplicity we set $G = \nb\times\Gmm B$ and write $G = (g^x,g^y,g^z)$. From \begin{equation*}
	\begin{split}
		\rd_tB + B\cdot\nb G - G\cdot\nb B = 0, 
	\end{split}
\end{equation*} taking the $z$-component gives \begin{equation*}
	\begin{split}
		\rd_t b^z - G \cdot\nb b^z + B\cdot\nb g^z = 0. 
	\end{split}
\end{equation*} We now note that \begin{equation*}
	\begin{split}
		B_h = \nb^\perp j,\qquad G_h = -\nb^\perp \Gmm b^z,\qquad g^z = \Gmm\lap j. 
	\end{split}
\end{equation*} ($V_h$ stands for the vector consisting of the first two components of $V$. Moreover, $\nb^\perp = (-\rd_y, \rd_x)$.) Then, the equation for $b^z$ can be written as \begin{equation*}
	\begin{split}
		\rd_t b^z + \nb^\perp\Gmm b^z \cdot \nb b^z + \nb^\perp j \cdot \nb \Gmm\lap j  = 0 . 
	\end{split}
\end{equation*}
Next, from  \begin{equation*}
	\begin{split}
		\rd_t b^x - G \cdot\nb b^x + B\cdot\nb g^x = 0,  
	\end{split}
\end{equation*}  noting that \begin{equation*}
	\begin{split}
		- G \cdot\nb b^x + B\cdot\nb g^x &= -g^x\rd_xb^x - g^y\rd_yb^x + b^x\rd_x g^x + b^y\rd_y g^x  \\
		& = \rd_y(g^xb^y) - \rd_y(g^yb^x)
	\end{split}
\end{equation*} (we have used the divergence-free conditions $\rd_xb^x+\rd_yb^y = 0$, $\rd_xg^x+\rd_yg^y=0$), we have that \begin{equation*}
	\begin{split}
		\rd_t j + \nb^\perp\Gmm b^z\cdot\nb j = 0. 
	\end{split}
\end{equation*} We have arrived at the closed system for $b^z$ and $j$: \begin{equation}\label{eq:EH-reduced2}
	\left\{
	\begin{aligned}
		&\rd_t b^z + \nb^\perp\Gmm b^z \cdot \nb b^z + \nb^\perp j \cdot \nb \Gmm\lap j  = 0 , \\
		&	\rd_t j + \nb^\perp\Gmm b^z\cdot\nb j = 0. 
	\end{aligned}
	\right.
\end{equation} In particular, note that when $j_{0}=0$,  then the second equation of \eqref{eq:EH-reduced2} implies that $j(\cdot, t)=0$ for $t>0$
as long as the smooth solution persists. Then, we have simply 
\begin{equation*}
	\begin{split}
		\rd_t b^z + \nb^\perp\Gmm b^z \cdot \nb b^z  = 0,
	\end{split}
\end{equation*} which is (gSQG).
In this sense the (gSQG) systems are embedded into \eqref{eq:EH}.  In the case of $\Gmm=\lap^{-1}$, this is nothing but the 2D Euler equations.
}

\section{ {Proof of  the main theorem}}\label{sec:proof}

We first prove the the following key commutator estimate.

\begin{lemma}\label{lem:comm-1} Let $\Lmb=(-\lap)^{\frac12}$ and $\Gmm$ satisfy the assumptions \eqref{as1}--\eqref{as3}. For $g\in L^2 (\Bbb R^n)$ and ${b}\in \calY^1$, we have \begin{equation}\label{eq:comm-1}
		\begin{split}
			\nrm{\mathbf{P}_{\ge 1} \Lmb^{\frac12} [\Gmm^{\frac12},b]\nb g}_{L^{2}} \le C\nrm{b}_{\calY^1} \nrm{g}_{L^{2}}
		\end{split}
	\end{equation} and \begin{equation}\label{eq:comm-2}
	\begin{split}
			\nrm{\Lmb^{\frac12}\mathbf{P}_{<1} \Lmb^{\frac12} [\Gmm^{\frac12},b]\nb g}_{L^{2}} \le C\nrm{b}_{\calY^1} \nrm{g}_{L^{2}}.
	\end{split}
\end{equation} Here, $\mathbf{P}_{\ge 1}$ is defined with Fourier multiplier $\mathbf{1}_{|\xi|\ge1}$ and $\mathbf{P}_{<1}=1-\mathbf{P}_{\ge 1}$.
\end{lemma}

\begin{proof}
	Taking the Fourier transform of $f:=\Lmb^{\frac12} [\Gmm^{\frac12},b]\nb g$, we have \begin{equation*}
		\begin{split}
		\widehat{f}(\xi)=	|\xi|^{\frac12} \int_{ \Bbb R^n}  \left( \gmm^{\frac12}(|\xi|) - \gmm^{\frac12}(|\eta|) \right) \widehat{b}(\xi-\eta)   i\eta\widehat{g}(\eta) \, \ud\eta. 
		\end{split}
	\end{equation*}  We bound, using the assumptions for $\gmm$, \begin{equation*}
	\begin{split}
		|\gmm^{\frac12}(|\xi|) - \gmm^{\frac12}(|\eta|)| &=\frac{|\gmm(|\xi|) - \gmm(|\eta|)|}{\gmm^{\frac12}(|\xi|) + \gmm^{\frac12}(|\eta|)}  {=}  \frac{1}{ \gmm^{\frac12}(|\xi|)+ \gmm^{\frac12}(|\eta|) } \int_{|\xi|}^{|\eta|} -\gmm'( {\rho})\,\ud 
		 {\rho} \\
		&\le    \int_{|\xi|}^{|\eta|} C\rho^{-1} \frac{\gmm(\rho)}{{ \gmm^{\frac12}(|\xi|)+ \gmm^{\frac12}(|\eta|) }} \,\ud 
		{\rho}   \le  C \int_{|\xi|}^{|\eta|} {\rho}^{-1}\gmm^{\frac12}( {\rho})\,\ud  {\rho}\\
		&\le C|\xi-\eta| \left( \frac{1}{(|\xi||\eta|)^{\frac12} (|\xi|^{\frac12} + |\eta|^{\frac12}) } + \frac{1}{|\xi||\eta|} \right) \\
			& \le \frac{C|\xi-\eta|}{|\eta|}\left( |\xi|^{-\frac12} + |\xi|^{-1} \right). 
	\end{split}
\end{equation*}  Here, we  have assumed $|\xi|\le|\eta|$, {and used both of \eqref{as2} and \eqref{as1}}, but the same inequality can be trivially deduced when $|\xi|>|\eta|$ as well. Inserting this bound into the integral expression for $\widehat{f}(\xi)$, we have that \begin{equation*}
\begin{split}
	\left| \widehat{f}(\xi) \right| \le C  {(1+|\xi|^{-\frac12})} \int_{ \Bbb R^n} |\xi-\eta||\widehat{b}(\xi-\eta)|   |\widehat{g}(\eta)| \, \ud\eta. 
\end{split}
\end{equation*} 
By Young's convolution inequality, we obtain   {
$$
	\nrm{\mathbf{P}_{\ge 1} f}_{L^2}= C\nrm{ \mathbf{1}_{|\xi|\ge 1} \widehat{f}}_{L^{2}} \le C \|b\|_{\calY^1} \|g\|_{L^{2}}
$$	 and \begin{equation*}
	\begin{split}
		\nrm{\Lmb^{\frac12}\mathbf{P}_{< 1} f}_{L^2}= C\nrm{ |\xi|^{\frac12}\mathbf{1}_{|\xi|<1} \widehat{f}}_{L^{2}} \le C \|b\|_{\calY^1} \|g\|_{L^{2}}. 
	\end{split}
\end{equation*} } This finishes the proof. 
\end{proof}

 { When the operator $\Gmm$ is given precisely by $\Lmb^{-a}$, we can obtain the following sharp inequality, with a straightforward modification of the above proof.
\begin{corollary} \label{lem:comm-3} For $g\in L^2 (\Bbb R^n)$ and ${b}\in \dot{\calY}^{1}(\bbR^{n})$, we have \begin{equation}\label{eq:comm-3}
			\begin{split}
				\nrm{\Lmb^{\frac{a}{2}} [\Lmb^{-\frac{a}{2}},b]\nb g}_{L^{2}} \le C\nrm{b}_{ \dot{\calY}^{1} } \nrm{g}_{L^{2}}, \quad \nrm{b}_{\dot{\calY}^{1}} := \int_{\bbR^n} |\xi||\widehat{b}(\xi)| \,\ud\xi. 
			\end{split}
		\end{equation}
\end{corollary}
}

\begin{lemma}\label{lem:comm-2}
The following embedding relations hold.
\be\label{emb}
 H^s (\Bbb R^n)\hookrightarrow  \calY^1 (\bbR^n) \hookrightarrow     \dot{W}^{1, \infty} (\Bbb R^n)
\ee
 for $s>\frac n2 +1.$  
 \end{lemma}
 \begin{proof}
 For all $x\in \Bbb R^n$ the following estimates hold.
\begin{equation*}
\begin{split}	
| \nabla f (x)| &\le \int_{\Bbb R^n} |\xi| |\hat{b}(\xi)|\ud\xi  {\le}	\|f\|_{ \calY^1}  \le   \int_{ \Bbb R^n} ( 1+|\xi|^2) ^{\frac{s}{2}} |\widehat{f}(\xi)|  (	1+|\xi|^2 ) ^{-\frac{ s-1}{2}} \ud\xi \\
	&\le \left(  \int_{ \Bbb R^n} ( 1+|\xi|^2) ^{s} |\widehat{f}(\xi)|^2 \ud\xi \right)^{\frac12} \left(  \int_{ \Bbb R^n} (	1+|\xi|^2 ) ^{-s+1} \ud\xi	\right)^{\frac12} \le C \nrm{f}_{H^{s}} ,
\end{split}
\end{equation*} 
where we used the fact $ \int_{ \Bbb R^n} (	1+|\xi|^2 ) ^{-s+1} \, \ud\xi<+\infty $ for $s>\frac n2 +1$.
Hence, we have shown $ \| \nabla f\|_{L^\infty} \le \|f\|_{\calY^{1}}  \le C \|f\|_{H^s} $, and the lemma is proved.
\end{proof} 

Lastly, we shall need a simple commutator estimate. 
\begin{lemma}\label{lem:comm-4} For two functions $f, g$ belonging to $H^s$, we have 
	\begin{equation*}
		\begin{split}
			\nrm{ [\Lmb^{s}, f\cdot\nb] g }_{L^{2}} \le C(\nrm{f}_{\calY^{1}} \nrm{g}_{H^{s}} + \nrm{g}_{\calY^{1}} \nrm{f}_{H^{s}} ). 
		\end{split}
	\end{equation*}
\end{lemma}
\begin{proof}
	Note that the Fourier transform of $H:=[\Lmb^{s}, f\cdot\nb] g$ is given by \begin{equation*}
			\begin{split}
				\widehat{H}(\xi)= \int (|\xi|^{s}-|\eta|^{s}) \widehat{f}(\xi-\eta) i\eta \widehat{g}(\eta)  \,\ud\eta. 
			\end{split}
		\end{equation*} Proceeding similarly with the proof of Lemma \ref{lem:comm-1}, we obtain that \begin{equation*}
			\begin{split}
				\left| \widehat{H}(\xi) \right| \le  C \int (|\xi-\eta|^{s-1}+|\eta|^{s-1}) |\xi-\eta||\widehat{f}(\xi-\eta)| |\eta \widehat{g}(\eta)|  \,\ud\eta 
			\end{split}
		\end{equation*} which gives with Young's convolution inequality that \begin{equation*}
			\begin{split}
				\nrm{H}_{L^2} \le C(\nrm{f}_{\calY^{1}} \nrm{g}_{H^{s}} + \nrm{g}_{\calY^{1}} \nrm{f}_{H^{s}} ).
			\end{split}
	\end{equation*} This gives the estimate. 
\end{proof}

\begin{proof}[Proof of Theorem \ref{thm:lwp}]
In the proof below we shall find the following vector calculus identity useful: 
\begin{equation}\label{eq:curl-times}
	\begin{split}
		\nb \times (B \times F) = ( \nb\cdot F + F \cdot \nb) B - (\nb\cdot B + B\cdot\nb) F. 
	\end{split}
\end{equation}	
We compute 
\begin{equation*}
		\begin{split} 
			\frac12 \frac{\ud}{\ud t} \nrm{\Lmb^sB}_{L^{2}}^{2} &= - \int_{ \Bbb R^3}  \Lmb^sB \cdot \Lmb^s \nb\times ((\nb\times\Gmm B)\times B) \, \ud x 
		\end{split}
	\end{equation*}
	The main contribution on the right hand side comes from \begin{equation*}
		\begin{split}
			I_1 := - \int_{ \Bbb R^3}  \Lmb^sB \cdot \nb\times ((\nb\times\Gmm \Lmb^s B)\times B) \, \ud x 
		\end{split}
	\end{equation*} and \begin{equation*}
	\begin{split}
		I_2:= - \int_{ \Bbb R^3}  \Lmb^sB \cdot\nb\times ((\nb\times\Gmm B)\times  \Lmb^s B) \, \ud x .
	\end{split}
\end{equation*} We first treat $I_{1}$: defining $F = \nb\times \Lmb^sB$ for simplicity, we have that \begin{equation*}
\begin{split}
	I_1 &= - \int_{ \Bbb R^3}  F\cdot (\Gmm F\times B) \, \ud x = \int_{ \Bbb R^3}  \epsilon_{ijk} \Gmm^{\frac12}(F^{i} B^{k})  \, \Gmm^{\frac12} F^{j}\,\ud x \\
	& = \int_{ \Bbb R^3}  \epsilon_{ijk} \Lmb^{\frac12} [\Gmm^{\frac12},B^{k}]F^{i} \,  \Lmb^{-\frac12} \Gmm^{\frac12} F^{j}\,\ud x, 
\end{split}
\end{equation*} 
 {where $ \epsilon_{ijk} $ is the totally skew-symmetric tensor with the normalization $\epsilon_{123}=1$, and 
we used }
 \begin{equation*}
\begin{split}
	 \int_{ \Bbb R^3}   \epsilon_{ijk} B^{k} \Gmm^{\frac12}(F^{i}) \, \Gmm^{\frac12} (F^{j})\,\ud x  = 0
\end{split}
\end{equation*} (repeated indices are being summed).  { We write \begin{equation*}
\begin{split}
	\int_{ \Bbb R^3}   \Lmb^{\frac12} [\Gmm^{\frac12},B^{k}]F^{i} \,  \Lmb^{-\frac12} \Gmm^{\frac12} F^{j}\,\ud x & = 	\int_{ \Bbb R^3} \mathbf{P}_{\ge 1}  \left( \Lmb^{\frac12} [\Gmm^{\frac12},B^{k}]F^{i} \right)\,  \Lmb^{-\frac12} \Gmm^{\frac12} F^{j}\,\ud x  \\
	&\qquad + 	\int_{ \Bbb R^3}  \Lmb^{\frac12} \mathbf{P}_{<1} \left(  \Lmb^{\frac12} [\Gmm^{\frac12},B^{k}]F^{i}\right) \,  \Lmb^{-1} \Gmm^{\frac12} F^{j}\,\ud x 
\end{split}
\end{equation*} }Applying Lemma \ref{lem:comm-1},  {and using the assumption \eqref{as1}},  we have that 
\begin{equation*}
\begin{split}
	&\left| \int_{ \Bbb R^3}   \epsilon_{ijk} \Lmb^{\frac12} [\Gmm^{\frac12},B^{k}]F^{i} \,  \Lmb^{-\frac12} \Gmm^{\frac12} F^{j}\,\ud x \right| \\
	&\qquad \le  {C\sum_{i,j,k}\left(\nrm{ \mathbf{P}_{\ge1}  \Lmb^{\frac12} [\Gmm^{\frac12},B^{k}]F^{i}}_{L^{2}}\nrm{\Lmb^{-\frac12} \Gmm^{\frac12} F^{j}}_{L^2} + \nrm{ \Lmb^{\frac12} \mathbf{P}_{<1}  \Lmb^{\frac12} [\Gmm^{\frac12},B^{k}]F^{i}}_{L^{2}}\nrm{\Lmb^{-1} \Gmm^{\frac12} F^{j}}_{L^2}\right)} \\
			& \qquad  \le C { \nrm{B}_{  \calY^1}  }  \| \Lmb^s B\|_{L^2}   {\left(    \nrm{ \Lmb^{-\frac12} \Gmm^{\frac12}\nb\times \Lmb^sB }_{L^2} +  \nrm{ \Lmb^{-1} \Gmm^{\frac12}\nb\times \Lmb^sB }_{L^2} \right) }\\
			&\qquad \le C {\nrm{B}_{  \calY^1} }\nrm{B}_{H^{s}}^{2}. 
\end{split}
\end{equation*} We now treat $I_{2}$: using \eqref{eq:curl-times}, 
\begin{equation*}
\begin{split}
	I_2 &=  - \int_{ \Bbb R^3}  \Lmb^s B \cdot (\Lmb^s B \cdot \nb (\nb\times\Gmm B)) \, \ud x + \int_{ \Bbb R^3}  \Lmb^s B \cdot ((\nb\times\Gmm B)\cdot\nb) \Lmb^s B \, \ud x \\
	&=I_{2,a} +I_{2,b}.
\end{split}
\end{equation*} The term $I_{2,a}$ can  be directly bounded by \begin{equation*}
\begin{split}
	I_{2,a}\le C\nrm{ \nb^{2}\Gmm B }_{L^{\infty}} \nrm{B}_{H^{s}}^{2} \le C {\nrm{B}_{  \calY^1 } \| B\|_{{H^{s}}}^2},
\end{split}
\end{equation*} using  {\eqref{as1}} and $s> \frac{5}{2}$. For the second term, we note that \begin{equation*}
\begin{split}
	I_{2,b} = \int_{ \Bbb R^3}   ((\nb\times\Gmm B)\cdot\nb) \frac{|\Lmb^s B|^{2}}{2} \, \ud x  = - \int_{ \Bbb R^3}  (\nb\cdot (\nb\times\Gmm B) ) \frac{|\Lmb^s B|^{2}}{2} \, \ud x = 0.  
\end{split}
\end{equation*}

Finally, we treat the remainder.  Using the identity \eqref{eq:curl-times},  we observe that 
\begin{equation}\label{eq:key}
	\begin{split}
		& I_1+ I_2+ \int_{ \Bbb R^3}  \Lmb^s B \cdot \Lmb^s \nb\times ((\nb\times\Gmm B)\times B) \, \ud x \\
		&\quad = \int_{ \Bbb R^3}  \Lmb^s B \cdot\left\{ \Lmb^s( \nb\times(G\times B)) - \nb\times(\Lmb^sG\times B) - \nb\times(G\times \Lmb^s B)  \right\}\,\ud x  \\
		& \quad=\int_{ \Bbb R^3}  \Lmb^s B 	\cdot\left\{   \Lmb^s( B\cdot\nb G - G\cdot\nb B ) - B\cdot\nb \Lmb^sG + \Lmb^sG\cdot\nb B - \Lmb^s B\cdot\nb G + G\cdot\nb \Lmb^s B\right\} dx\\
		&\quad=\int_{ \Bbb R^3}  \Lmb^s B 	\cdot  	\{ [\Lmb^s, B\cdot\nb ] G \} dx
		-\int_{ \Bbb R^3}  \Lmb^s B 	\cdot  \{ [\Lmb^s,G\cdot\nb] B \} dx   \\
		&\qquad\qquad - \int_{ \Bbb R^3}  \Lmb^s B 	\cdot  \{(\Lmb^s B\cdot\nb ) G\} dx  + \int_{ \Bbb R^3}  \Lmb^s B 	\cdot  \{ \Lmb^{s} G \cdot  \nb B\} dx\\
		&\quad:=K_1+K_2+K_3+K_4,
			\end{split}
\end{equation}  where   we set $G := \nb\times\Gmm B$. We note that \begin{equation*}
\begin{split}
	\nrm{\Lmb^s G}_{L^2} \le C \nrm{B}_{H^{s}}, \qquad \nrm{\nb G}_{L^\infty} \le C\nrm{B}_{{\calY}^1}. 
\end{split}
\end{equation*} To see the first inequality, \begin{equation*}
\begin{split}
	\nrm{\Lmb^s G}_{L^2}^{2} & \le C\int |\xi|^{2s} |\xi\times \gmm(|\xi|) \hat{B}(\xi) |^{2} \ud \xi \\
	& \le C\int |\xi|^{2s+2} (|\xi|^{-1}+|\xi|^{-2})^{2} |\hat{B}(\xi) |^{2} \ud \xi \le C\nrm{B}_{H^{s}}^{2},
\end{split}
\end{equation*} using the assumption \eqref{as1} for $\gmm$. The proof of the other inequality is similar.

Then, we can bound $K_4$ by 
$$
\left|K_4\right| \le C \nrm{\Lmb^sB}_{L^2} \nrm{\Lmb^sG}_{L^2} \nrm{\nb B}_{L^\infty} \le C  \| B \|_{     \dot{\calY}^1 } \nrm{B}_{H^{s}}^{2}.
$$
Next, we estimate $K_3$ easily as follows:
\begin{equation*}
\begin{split}
	\left|K_3\right|\le   \int_{ \Bbb R^3} | \Lmb^s B \cdot \{  \Lmb^s B\cdot\nb G \} |\,\ud x      \le C \nrm{ \nb G }_{L^\infty}\nrm{B}_{H^{s}}^{2} \le  C \| B \|_{     \dot{\calY}^1 } \nrm{B}_{H^{s}}^{2}.
\end{split}
\end{equation*}  Lastly, using Lemma \ref{lem:comm-4}, we can directly estimate the first two terms by \begin{equation*}
\begin{split}
\left|K_1+K_2\right| &\le  \|  \Lmb^s B \|_{L^2} 	\left(\nrm{  [\Lmb^s, B\cdot\nb ] G }_{L^2} + \nrm{  [\Lmb^s,G\cdot\nb] B  }_{L^{2}} \right) \\
&\le    C  { (  \| B \|_{\calY^{1}} + \| G \|_{\calY^{1}} )}  (  \nrm{B}_{H^{s}} + \nrm{G}_{H^{s}} )  \|  \Lmb^s B \|_{L^2} \\
&\le C \| B\|_{\calY^{1}} \| B \|_{H^s}^2.
\end{split}
\end{equation*}
This finishes the proof of 
  \begin{equation*}
  \label{locest}\begin{split}
	 \frac{\ud}{\ud t} \nrm{ B}_{ {\dot{H}^s}}\le C \| B\|_{\calY^{1}}  \nrm{B}_{H^{s}}. 
\end{split}
\end{equation*}  {On the other hand, from \eqref{eq:EH-active}, it is straightforward to obtain \begin{equation*}
\begin{split}
	\frac{\ud}{\ud t} \nrm{B}_{L^{2}} \le C \nrm{\nb^2\Gmm[B]}_{L^{\infty}}\nrm{B}_{L^{2}} \le C \| B\|_{\calY^{1}}  \nrm{B}_{H^{s}}.
\end{split}
\end{equation*} Combining the above estimates together with \eqref{emb},} we have
  \begin{equation*}
  \label{locest1}\begin{split}
	 \frac{\ud}{\ud t} \nrm{ B}_{H^s} \le C \nrm{B}_{H^{s}}^2. 
\end{split}
\end{equation*} Hence, there exists $T>0$ depending only on $\nrm{B_{0}}_{H^{s}}$, such that \begin{equation}\label{eq:apriori}
\begin{split}
	\sup_{t\in[0,T]}\nrm{B(t,\cdot)}_{H^{s}} \le 2\nrm{B_0}_{H^{s} }.
\end{split}
\end{equation} 

Given the a priori estimate, it is a simple matter to prove existence and uniqueness of a solution belonging to $L^\infty([0,T];H^{s} )$. For uniqueness, one may assume that there are two solutions $B$ and $\tilde{B}$ to \eqref{eq:EH} with the same initial data $B_{0}$, and that $B, \tilde{B} \in L^\infty([0,T];H^{s})$ for some $T>0$. Proceeding similarly as in the proof of \eqref{eq:apriori}, using the equation satisfied by the difference $B - \tilde{B}$, one can prove \begin{equation*}
	\begin{split}
		\frac{d}{dt} \nrm{ B - \tilde{B} }_{L^{2}} \le C( \nrm{B}_{L^\infty([0,T];H^{s})} + \nrm{\tilde{B}}_{L^\infty([0,T];H^{s} )} ) \nrm{ B - \tilde{B} }_{L^{2}} . 
	\end{split}
\end{equation*} (See \cite{CJO2,CJ3} for details of this argument for closely related systems.) For existence, we consider the regularized systems \begin{equation}\label{eq:EH-v}
\left\{
\begin{aligned}
	& \rd_t B + \nb\times ( ( \nb\times \Gmm [B] ) \times B) - \nu \lap B = 0 ,  \\
	& \nabla \cdot B = 0, \\
	& B(t=0) = B_0, 
\end{aligned}
\right.
\end{equation} for each $\nu>0$. Given $B_0 \in H^{s}$, it is straightforward to show local existence of a solution $B^{(\nu)}$ to \eqref{eq:EH-v} satisfying $B^{(\nu)} \in L^\infty([0,T_\nu];H^{s})$ for some $T_\nu>0$, using the mild formulation; one can follow the arguments of \cite{CDL}. Since the a priori estimate \eqref{eq:apriori} applies to the viscous solutions $B^{(\nu)}$, the lifespan $T_\nu$ can be bounded from below by some $T>0$ independent of $\nu>0$. There exists a weak limit $B^{(\nu)} \to B$ in $L^\infty([0,T];H^{s})$ (by taking a sub-sequence if necessary), and the limit is a solution to \eqref{eq:EH}. 

Lastly, we argue that the solution $B$ actually belongs to $C([0,T];H^{s})$. Using time-translation and time-reversal symmetries, it suffices to prove strong convergence $B(t)\to B_0$ for $t\to 0^+$. To begin with, from uniform boundedness, we obtain weak convergence; this gives that \begin{equation*}
	\begin{split}
		\nrm{B_0}_{ H^s} \le \liminf_{t\to 0^+} \nrm{B(t)}_{ H^s }. 
	\end{split}
\end{equation*} On the other hand, from \eqref{locest1}, we see that 
\begin{equation*}
\begin{split}
	\nrm{B_0}_{ H^s  } - Ct \le \nrm{B(t)}_{{ H^s }} \le \nrm{B_0}_{ H^s } + Ct, \quad \forall t \in [0,T_0]
\end{split}
\end{equation*} for some $T_0, C>0$ depending only on $\nrm{B_0}_{ H^s  }$. This gives norm convergence $\nrm{B(t)}_{ H^s } \to \nrm{B_0}_{ H^s }$, which together with weak convergence concludes strong convergence. The proof of (i) is now complete. \\
The proof of (ii) immediately  follows from the inequality
$$
\|B(t) \|_{H^s} \le \|B_0\|_{H^s} e^{C\int_0 ^t \|B(s) \|_{\calY^{1}} ds },
$$
which is obtained from \eqref{locest}. \end{proof}

 {
\begin{proof}[Proof of Theorem B]
	We shall simply establish an $H^{s}$--a priori estimate for \eqref{eq:EH-diss}, and omit the proof of existence and uniqueness. Similarly as in the proof of Theorem A, we compute 
	\begin{equation*}
		\begin{split} 
			\frac12 \frac{\ud}{\ud t} \nrm{\Lmb^sB}_{L^{2}}^{2} + \nrm{\Lmb^{s+\frac{b}{2}} B }_{L^{2}}^{2} &= - \int_{ \Bbb R^3}  \Lmb^sB \cdot \Lmb^s \nb\times ((\nb\times\Lmb^{-a} B)\times B) \, \ud x . 
		\end{split}
	\end{equation*}
	The main terms on the right hand side are given by\begin{equation*}
		\begin{split}
			I_1 := - \int_{ \Bbb R^3}  \Lmb^sB \cdot \nb\times ((\nb\times\Lmb^{-a} \Lmb^s B)\times B) \, \ud x 
		\end{split}
	\end{equation*} and \begin{equation*}
		\begin{split}
			I_2:= - \int_{ \Bbb R^3}  \Lmb^sB \cdot\nb\times ((\nb\times\Lmb^{-a} B)\times  \Lmb^s B) \, \ud x .
		\end{split}
	\end{equation*}  With $F := \nb\times \Lmb^sB$, recall that \begin{equation*}
		\begin{split}
			I_1 = \sum_{i,j,k} \int_{ \Bbb R^3}  \epsilon_{ijk}  [\Lmb^{-\frac{a}{2}},B^{k}] F^{i} \, \Lmb^{-\frac{a}{2}} F^{j}\,\ud x = \sum_{i,j,k} \int_{ \Bbb R^3}  \epsilon_{ijk} \Lmb^{\frac{1}{2}} [\Lmb^{-\frac{a}{2}},B^{k}] F^{i} \, \Lmb^{- \frac{a+1}{2}} F^{j}\,\ud x, 
		\end{split}
	\end{equation*} which can be bounded by \begin{equation*}
	\begin{split}
		\left|I_1\right|\le C \nrm{ \Lmb^{\frac{1}{2}} [\Lmb^{-\frac{a}{2}},B^{k}] \Lmb^{\frac{a+1}{2}} }_{L^2\to L^2} \nrm{\Lmb^{- \frac{a+1}{2}} F^{j}}^{2}_{L^{2}}. 
	\end{split}
\end{equation*} To bound the operator norm of the commutator, observe that \begin{equation*}
\begin{split}
	 \Lmb^{\frac{1}{2}} [\Lmb^{-\frac{a}{2}},B^{k}] \Lmb^{\frac{a+1}{2}}  = \Lmb^{\frac{1-a}{2}}[B, \Lmb^{\frac{1+a}{2}}] - \Lmb^{\frac{1}{2}}[B,\Lmb^{\frac{1}{2}}]
\end{split}
\end{equation*} and \begin{equation*}
\begin{split}
	\nrm{\Lmb^{\frac{1-a}{2}}[B, \Lmb^{\frac{1+a}{2}}] }_{L^2\to L^2} + \nrm{\Lmb^{\frac{1}{2}}[B,\Lmb^{\frac{1}{2}}]}_{L^2\to L^2} \le C\nrm{B}_{\dot{\calY}^{1}}. 
\end{split}
\end{equation*} This estimate can be proved easily using the Fourier transform and following the arguments of Lemma \ref{lem:comm-1} above. Therefore, \begin{equation*}
\begin{split}
	\left|I_1\right|\le C \nrm{B}_{\dot{\calY}^{1}} \nrm{ \Lmb^{s+ \frac{1-a}{2}} B}_{L^2}^2\le C \nrm{B}_{H^{s}} \nrm{ \Lmb^{s+ \frac{1-a}{2}} B}_{L^2}^2.
\end{split}
\end{equation*} Next, recall that with a cancellation, we have simply  
	\begin{equation*}
		\begin{split}
			I_2 &=  - \int_{ \Bbb R^3}  \Lmb^s B \cdot (\Lmb^s B \cdot \nb (\nb\times\Lmb^{-a} B)) \, \ud x. 
		\end{split}
	\end{equation*} Then, it is not difficult to obtain the bound \begin{equation*}
		\begin{split}
			\left| I_{2} \right| \le  C \nrm{\nb^2 \Lmb^{-a}B}_{L^{\infty}} \nrm{B}_{\dot{H}^{s}}^{2} \le C \nrm{B}_{H^{s}}^{3}, 
		\end{split}
	\end{equation*} where we have used that $s> \frac{3}{2}+2-a$. Finally, proceeding in the same way as \eqref{eq:key}, it is not difficult to prove the following estimate for the remainder: \begin{equation*}
	\begin{split}
		\left| I_1+ I_2+ \int_{ \Bbb R^3}  \Lmb^s B \cdot \Lmb^s \nb\times ((\nb\times \Lmb^{-a} B)\times B) \, \ud x \right| \le C \nrm{B}_{H^{s}}^{3}. 
	\end{split}
\end{equation*} Therefore, we conclude that \begin{equation*}
\begin{split}
	\frac12 \frac{\ud}{\ud t} \nrm{\Lmb^sB}_{L^{2}}^{2} + \nrm{\Lmb^{s+\frac{b}{2}} B }_{L^{2}}^{2} & \le C\nrm{B}_{{H}^{s}}\nrm{\Lmb^{s+\frac{1-a}{2}}B}_{L^2}^{2} +  C\nrm{B}_{H^{s}}^{3}  \le C\nrm{B}_{{H}^{s}}^{1+2\alp}\nrm{\Lmb^{s+\frac{b}{2}}B}_{L^2}^{2(1-\alp)} +  C\nrm{B}_{H^{s}}^{3}
\end{split}
\end{equation*} for some $0<\alp<1$ since $1-a<b$. Using Young's inequality, we can absorb the $\nrm{\Lmb^{s+\frac{b}{2}}B}_{L^2}$--term to the left hand side. Together with an $L^2$ estimate for $B$, this finishes the proof of the a priori estimate in $H^{s}$. 
\end{proof}
}

\section{{Discussion and open problems}}\label{subsec:discuss}

\medskip

\noindent \textbf{Extensions}. The proof of local well-posedness can be adapted to three-dimensional domains without boundary $\bbR^{k} \times \bbT^{3-k}$  for any $0\le k <3$. Moreover, one may consider the systems with a (fractional) dissipation term. It could be interesting to consider the \textit{critically dissipative} case, as an analogy with critically dissipative SQG equations studied for instance in \cite{crit1,crit2,crit3,crit4}. 

\medskip

\noindent  {\textbf{Propagation of $\dot{H}^{-1}$}. Given the formulation of the system in terms of $u$ and scaling of the conservation law, it is natural to add the assumption that $B \in \dot{H}^{-1}$. Indeed, $\nrm{B}_{\dot{H}^{-1}}$ is equivalent with $\nrm{u}_{L^2}$, and it is easy to see using the equation for $u$ that as long as $B \in L^\infty_t H^{s}$ with $s>\frac{5}{2}$, \begin{equation*}
		\begin{split}
			\frac{\ud}{\ud t} \nrm{B}_{\dot{H}^{-1}} \lesssim \nrm{B}_{\dot{H}^{-1}}. 
		\end{split}
	\end{equation*} That is, the assumption $B_{0} \in \dot{H}^{-1}$ propagates in time as long as the solution remains smooth. 
}

\medskip

\noindent \textbf{Blow-up criteria}. 
We remark that the blow-up criterion can be replaced with 
\begin{equation}\label{eq:blowup-2}
	\begin{split}
		\lim_{t\to T}\int_0^{t} \nrm{B(\tau,\cdot)}_{H^{\frac{5}{2}}}\, \ud\tau = +\infty. 
	\end{split}
\end{equation}  {To see this, we observe the logarithmic Sobolev estimate for $s>\frac{5}{2}$ 
\begin{equation}
\label{logsov}\begin{split}
	\nrm{B(\tau,\cdot)}_{\calY^{1}}\le C  \nrm{B(\tau,\cdot)}_{H^{\frac{5}{2}}} \log(10 +  \nrm{B(\tau,\cdot)}_{H^{s}}).
\end{split}
\end{equation} 
Then, one may obtain the a priori estimate 
\begin{equation*}
\begin{split}
	\frac{\ud}{\ud t} \nrm{B}_{H^{s}} \le C \nrm{B}_{H^{\frac{5}{2}}}  \nrm{B}_{H^{s}} \log(10 +  \nrm{B}_{H^{s}}), 
\end{split}
\end{equation*} from which the blow-up criterion \eqref{eq:blowup-2} {follows by applying the Gronwall lemma to the
differential inequality
$$ y^\prime   \le  a y\log y ,  \quad y=  10+\|B(t)\|_{H^s}, \quad a= C\|B(t)\|_{H^{\frac52}}.
$$ 
}}
{ The inequality \eqref{logsov}, in turn, can be verified as follows. We first estimate
\begin{align*}
 \int_{\Bbb R^3} |\xi| | \hat{f}| \ud\xi &\le \int_{\{ |\xi|<R\} }  (1+|\xi| ^2)^ {\frac54} | \hat{f}|  (1+|\xi| ^2)^ {-\frac34}  \ud\xi +   
\int_{\{ |\xi|\ge R\} }  (1+|\xi|^2 )^{\frac{s}{2}}  | \hat{f}|    (1+|\xi|^2 )^{\frac{1-s}{2}}  \ud\xi   \\
&  \le       \left( \int_{\{ |\xi|<R\} }  (1+|\xi| ^2)^ {\frac52} |   \hat{f}|^2 \ud\xi \right)^{\frac12} 
\left( \int_{\{ |\xi|<R\} }   (1+|\xi| ^2)^ {-\frac32}  \ud\xi \right)^{\frac12} \\
&\qquad + \left( \int_{\{|\xi| \ge R\}}   (1+|\xi|^2 )^{s}  | \hat{f}|^2 \ud\xi \right)^{\frac12}  \left( \int_{\{|\xi| \ge R\}}   (1+|\xi|^2 )^{1-s}  \ud\xi \right)^{\frac12 } \\
&\le C \| f\|_{H^{\frac52}} \log (1+R) +  C \| f\|_{H^s} R^{\frac52-s} \quad \forall R>0.
\end{align*}
Choosing $R= ( \|f\|_{H^s}/ \|f\|_{\frac52} )^{H^{\frac{2}{2s-5}}} $, we obtain \eqref{logsov}.
}

\medskip

\noindent \textbf{Issue of global well-posedness}. Global well-posedness of smooth solutions in the case of 3D Euler is a notoriously difficult open problem. In the electron--MHD case, finite time singularity formation is available for smooth and axisymmetric data (\cite{CWe}) but unfortunately local well-posedness (more precisely, uniqueness) is not known for the data used in the proof. It should be interesting to investigate the possibility of singularity formation for the interpolating models \eqref{eq:EH}. Of course, when the multiplier $\Gmm$ becomes ``very regular'', e.g. $\Gmm = (1 - \lap)^{-N}$ with $N$ large, then global well-posedness for smooth and decaying solutions can be {proved} using the conservation of energy.

\subsubsection*{Acknowledgments}{D.~Chae was supported partially  by NRF grant 2021R1A2C1003234. I.-J.~Jeong was supported by the NRF grant 2022R1C1C1011051.} We sincerely thank the anonymous referees for providing several helpful suggestions and pointing out errors as well as relevant references. 

\subsubsection*{Data availability statement} Data sharing not applicable to this article as no datasets were generated or analyzed during the current study.

\bibliographystyle{amsplain}


\end{document}